\documentclass{amsart}

\usepackage{amssymb, amsmath}
\usepackage{mathtools}
\usepackage{ytableau}
\usepackage{tikz}

\newcommand{\la}{\lambda}
\newcommand{\La}{\Lambda}
\newcommand{\we}{\wedge}

\DeclareMathOperator{\Hom}{Hom}

\DeclareMathOperator{\Sp}{\mathrm{Sp}}

\newtheorem{theorem}{Theorem}[section]
\newtheorem{lemma}[theorem]{Lemma}
\newtheorem{proposition}[theorem]{Proposition}
\newtheorem{cor}[theorem]{Corollary}
\theoremstyle{definition}

\newtheorem{example}[theorem]{Example}

\theoremstyle{remark}
\newtheorem{remark}[theorem]{Remark}
\newtheorem{remarks}[theorem]{Remarks}

\numberwithin{equation}{section}

\begin{document}

\title[homomorphisms into two part Weyl modules]{On homomorphisms into Weyl modules corresponding to partitions with two parts}


\author{Mihalis Maliakas}
\address{Department of Mathematics, University of Athens}

\email{mmaliak@math.uoa.gr}

\author{Dimitra-Dionysia Stergiopoulou}
\address{Department of Mathematics, University of Athens}

\email{dstergiop@math.uoa.gr}

\date{}

\begin{abstract}Let $K$ be an infinite field of characteristic $p>0$ and let $\lambda, \mu$ be partitions, where $\mu$ has two parts.  We find sufficient arithmetic conditions on $p, \lambda, \mu$ for the existence of a nonzero homomorphism $\Delta(\lambda) \to \Delta (\mu)$ of Weyl modules for the general linear group $GL_n(K)$. Also, for each $p$ we find sufficient conditions so that the corresponding homomorphism spaces have dimension at least 2.
\end{abstract}
\maketitle

\section{Introduction}
\label{intro}
In the representation theory of the general linear group $GL_n(K)$, where $K$ is an infinite field of characteristic $p>0$, the Weyl modules $\Delta(\la)$ are of central importance. These are parametrized by partitions $\la$ with at most $n$ parts. Over a field of characteristic zero, the modules $\Delta(\la)$ are irreducible. However over fields of positive characteristics this is no longer true and determining their structure is a major problem. In particular, very little is known about homomorphisms between them.

For $GL_3(K)$ all homomorphisms between Weyl modules have been classified when $p>2$ by Cox and Parker \cite{CoP}. Some of the few general results are the non vanishing theorems of Carter and Payne \cite{CP} and Koppinen \cite{Ko}, and the row or column removal theorems of Fayers and Lyle \cite{Ly1} and Kulkarni \cite{Ku}.

In \cite{MS} we examined homomorphisms into hook Weyl modules and obtained a classification result. This has been obtained also by Loubert \cite{Lou}. In the present paper we consider homomorphisms $\Delta(\la) \to \Delta(\mu)$, where $\mu$ has two parts. The main result, Theorem 3.1, provides sufficient arithmetic conditions on $\la, \mu$ and $p$ so that $\Hom_S(\Delta(\la), \Delta(\mu)) \neq 0$, where $S$ is the Schur algebra for $GL_n(K)$ of appropriate degree. An explicit map is provided that corresponds to the  sum of all standard tableaux of shape $\mu$ and weight $\la$. The main tool of the proof is the description of Weyl modules by generators and relations of Akin, Buchsbaum and Weyman \cite{ABW}.

The first examples of pairs of Weyl modules with homomorphism spaces of dimension greater than 1 were obtained by Dodge \cite{Do}. Shortly after, more were found by Lyle \cite{Ly1}. In Corollary 6.2, we find sufficient conditions on $\la, \mu$ and $p$ so that $\dim \Hom_S(\Delta(\la), \Delta(\mu))>1$ and thus we have new examples of homomorphism spaces between Weyl modules of dimension greater than 1.

By a classical theorem of Carter and Lusztig \cite{CL}, the results in Theorem 3.1 and Corollary 6.2 have analogues for Specht modules for the symmetric group when $p>2$, see Remark 3.2 and Remark 6.3.

Section 2 is devoted to notation and preliminaries. In Section 3 we state the main result and in Section 4 we consider the straightening law  needed later. The proof of the main result is in Section 5. In Section 6 we consider homomorphism spaces of dimension greater than 1.

\section{Preliminaries}
\subsection{Notation}

Throughout this paper, $K$ will be an infinite field of
characteristic $p>0$. We will be working with homogeneous polynomial representations of $GL_n(K)$ of degree $r$, or equivalently, with modules over the Schur algebra $S=S_K(n,r)$. A standard reference here is  \cite{Gr}.

In what follows we fix notation and recall from Akin and Buchsbaum \cite{AB}, and also Akin, Buchsbaum and Weyman \cite{ABW} important facts.

Let $V=K^n$ be the natural $GL_n(K)$-module. The divided power algebra $DV=\sum_{i\geq 0}D_iV$ of $V$ is defined as the graded dual of the Hopf algebra $S(V^{*})$, where $V^{*}$ is the linear dual of $V$ and $S(V^{*})$ is the symmetric algebra of $V^{*}$, see \cite{ABW}, I.4. For $v \in V$ and $i, j$ nonnegative integers, we will use many times relations of the form \[v^{(i)}v^{(j)}=\tbinom{i+j}{j}v^{(i+j)},\] where $\tbinom{i+j}{j}$ is the indicated binomial coefficient.

By $\wedge(n,r)$ we denote the set of sequences $a=(a_1, \dots, a_n)$ of nonnegative integers that sum to $r$ and by $\wedge^+(n,r)$ we denote the subset of $\wedge(n,r)$ consisting of sequences $\lambda=(\lambda_1, \dots, \lambda_n)$ such that $\lambda_1 \ge \lambda_2 \dots \ge \lambda_n$. Elements of $\wedge^+(n,r)$ are referred to as partitions of $r$ with at most $n$ parts. The transpose partition $\la^t =(\la_1^t,...,\la_n^t) \in \we^+(\la_1,r)$ of a partition $\la=(\la_1,...,\la_n) \in \we^+(n,r)$ is defined by $\la_j^t = \# \{i: \la_i \ge j\}$.

If $a=(a_1,\dots, a_n) \in \wedge(n,r)$, we denote by $D(a)$ or $D(a_1,\dots,a_n)$ the tensor product $D_{a_1}V\otimes \dots \otimes D_{a_n}V$. All tensor products in this paper are over $K$.

The exterior algebra of $V$ is denoted $\Lambda V=\sum_{i\geq 0}\Lambda^iV$. If $a=(a_1,\dots, a_n) \in \wedge(n,r)$, we denote by $\Lambda(a)$  the tensor product $\Lambda^{a_1}V\otimes \dots \otimes \Lambda^{a_n}V$.

For $\lambda \in \wedge^+(n,r)$, we denote by $\Delta(\lambda)$ the corresponding Weyl module for $S$. In \cite{ABW}, Definition II.1.4, the module $\Delta(\la)$ (denoted $K_{\la}F$ there), was defined as the image a map particular $d'_{\la} : D(\la) \to \Lambda(\la^{t})$. For example, if $\lambda =(r)$, then $\Delta(\lambda) =D_rV$, and if $\lambda =(1^r)$, then $\Delta(\lambda) =\La^{r}V$.

\subsection{Relations for Weyl modules.} We recall from \cite{ABW}, Theorem II.3.16, the following description of $\Delta (\lambda)$ in terms of generators and relations. \begin{theorem}[\cite{ABW}] Let $\lambda=(\lambda_1,\dots,\lambda_m) \in \wedge^+(n,r)$, where $\lambda_m >0$. There is an exact sequence of $S$-modules \[
\sum_{i=1}^{m-1}\sum_{t=1}^{\lambda_{i+1}}D(\lambda_1,\dots,\lambda_i+t,\lambda_{i+1}-t,\dots,\lambda_m) \xrightarrow{\square_{\la}} D(\lambda) \xrightarrow{d'_\lambda} \Delta(\lambda) \to 0,
\]
where the restriction of $ \square_{\la} $ to the summand $M(t)=D(\lambda_1,\dots,\lambda_i+t,\lambda_{i+1}-t,\dots,\lambda_m)$ is the composition
\[
M(t) \xrightarrow{1\otimes\cdots \otimes \Delta \otimes \cdots 1}D(\lambda_1,\dots,\lambda_i,t,\lambda_{i+1}-t,\dots,\lambda_m)\xrightarrow{1\otimes\cdots \otimes \eta \otimes \cdots 1} D(\lambda),
\]
where $\Delta:D(\lambda_i+t) \to D(\lambda_i,t)$ and $\eta:D(t,\lambda_{i+1}-t) \to D(\lambda_{i+1})$ are the indicated components of the comultiplication and multiplication respectively of the Hopf algebra $DV$ and $d'_\lambda$ is the map in \cite{ABW}, Def.II.13.\end{theorem}

\subsection{Standard basis of $\Delta(\mu)$} We will record here and in the next subsection two important facts from \cite{ABW} and \cite{AB} specified to the case of partitions consisting of two parts. 

Let us fix the order $e_1<e_2<...<e_n$ on the set $\{e_1, e_2, ..., e_n\}$ of the canonical basis elements of the natural module $V$ of $GL_n(K)$. We will denote each element $e_i$ by its subscript $i$. For a partition $\mu=(\mu_1,\mu_2) \in \we^+(n,r)$, a \textit{tableau} of shape $\mu$ is a filling of the diagram of $\mu$ with entries from $\{1,...,n\}$. Such a tableau is called \textit{standard} if the entries are weakly increasing across the rows from left to right and strictly increasing in the columns from top to bottom. (The terminology used in \cite{ABW} is 'co-standard'). 

The set of standard tableaux of shape $\mu$ will be denoted by $\mathrm{ST}(\mu)$. The \textit{weight} of a tableau $T$ is the tuple $\alpha=(\alpha_1,...,\alpha_n)$, where $\alpha_i$ is the number of appearances of the entry $i$ in $T$. The subset of $\mathrm{ST}(\mu)$ consisting of the (standard) tableaux of weight $\alpha$ will be denoted by $\mathrm{ST}_{\alpha}(\mu).$ 

For example, the following tableau of shape $\mu=(6,4)$. 
\begin{center}
$T=$
\begin{ytableau}
\ 1&1&1&2&2&4\\
\ 2&2&3&4
\end{ytableau}
\end{center}
is standard and has weight $\alpha=(3,4,1,2)$. 

We will use 'exponential' notation for standard tableaux. Thus for the above example we write \[ T=\begin{matrix*}[l]
1^{(3)}2^{(2)}4 \\
2^{(2)}34 \end{matrix*}. \]

To each tableau $T$ of shape $\mu=(\mu_1, \mu_2)$ we may associate an element \[x_T=x_T(1) \otimes x_T(2) \in D(\mu_1,\mu_2),\] where $x_T(i)=1^{(a_{i1})}\cdots n^{(a_{in})}$ and $a_{ij}$ is equal to the number of appearances of $j$ in the $i$-th row of $T$. For example, the $T$ depicted above yields $x_T=1^{(3)}2^{(2)}4\otimes2^{(2)}34$.
According to \cite{ABW}, Theorem II.2.16, we have the following. \begin{theorem}[\cite{ABW}]The set $\{d'_{\mu}(x_T): T \in \mathrm{ST}(\mu)\}$ is a basis of the $K$-vector space $\Delta({\mu})$.\end{theorem}

If $x=1^{(a_1)}2^{(a_2)}\cdots n^{(a_n)} \otimes 1^{(b_1)}2^{(b_2)}\cdots n^{(b_n)}  \in D(\mu)$, we will denote the element $d'_{\mu}(x) \in \Delta(\mu)$ by 	
\[ \begin{bmatrix*}
1^{(a_1)} 2^{(a_2)}   \cdots   n^{(a_n)} \\
1^{(b_1)} 2^{(a_2)}  \cdots  n^{(b_n)} \end{bmatrix*}. \]

\subsection{Weight subspaces of $\Delta(\mu)$} Let $\nu \in \we(n,r)$ and  $\mu=(\mu_1,\mu_2) \in \we^+(2,r)$. According to \cite{AB}, equation (11), 
a basis of the $K$-vector space  $\Hom_S(D(\nu), \Delta(\mu))$ is in 1-1 correspondence with set $\mathrm{ST}_{\nu}(\mu)$ of standard tableaux of shape $\mu$ and weight $\nu$.

For the computations to follow, we need to make the above correspondence explicit. Let $\nu = (\nu_1,..., \nu_n) \in \we(n,r)$ and $T \in \mathrm{ST}_{\nu}(\mu)$. Let $a_i$ (respectively, $b_i$) be the number of appearances of $i$ in the first row (respectively, second row) of $T$. We note that $\nu_i = a_i+b_i$ for each $i$.  In particular we have $a_1 = \nu_1$ because of standardness of $T$. Define the map \[\phi_T:D(\nu) \to \Delta(\mu),\]\[ x_1 \otimes x_2 \otimes \cdots\otimes x_n \mapsto \sum_{i_2,...,i_n}d'_{\mu}\left(x_1x_{2i_2}(a_2)\cdots x_{ni_n}(a_n) \otimes x_{2i_2}(b_2)'\cdots x_{ni_n}(b_n)'\right), \]
where $\sum_{i_s}x_{si_s}(a_s) \otimes x_{si_s}(b_s)'$ is the image of $x_s$  under the component \[D(\nu_s) \to D(a_s,b_s)\] of the diagonalization $\Delta : DV \to DV \otimes DV$ of the Hopf algebra $DV$ for $s=2,...,n$. Thus we have that a basis of the $K$-vector space  $\Hom_S(D(\nu), \Delta(\mu))$ is the set \[\{\phi_T: T \in \mathrm{ST}_{\nu}(\mu)\}.\]

In particular, suppose $\la=(\la_1,...,\la_m) \in \we^+(n,r)$ is a partition and $\mu =(\mu_1,\mu_2) \in \we^+(2,r)$ satisfies $\mu_2 \le \la_1$. This inequality means that each tableau of shape $\mu$ that has the form \[ \begin{matrix*}[l]
1^{(\la_1)} 2^{(a_2)}  \cdots   m^{(a_m)} \\
2^{(b_2)}  \cdots  m^{(b_m)}
\end{matrix*}\] is standard. Hence we have the following result.
\begin{lemma}
Let $\lambda,\mu \in \we^{+}(n,r)$, where $\lambda=(\lambda_1,...,\lambda_m)$ and $ \mu=(\mu_1, \mu_2)$. If $\mu_2 \le \la_1$, than a basis of the $K$- vector space $\Hom_S(D(\la, \Delta(\mu))$  is given by the elements $\phi_T$, where
\[[T]=	\begin{bmatrix*}[l]
1^{(\la_1)} 2^{(a_2)}  \cdots   m^{(a_m)} \\
2^{(b_2)}  \cdots  m^{(b_m)}
\end{bmatrix*},\] are such that 
\begin{align*} &a_i, b_i \ge 0, i=2,...,m\\
&a_i+b_i=\la_i, i=2,...,m,\\
&a_2+\cdots+a_m=\mu_1-\la_1,\\
&b_2+\cdots+b_m=\mu_2.
\end{align*}
\end{lemma}

\section{Main result}

In order to state the main result of this paper we use the following notation. If $x,y$ are positive integers, let \[R(x,y)=\gcd\{\tbinom{x}{1},  \tbinom{x+1}{2},...,\tbinom{x+y-1}{y}\}.\] If $x$ is a positive integer, let $R(x,0)=0$. 

\begin{theorem} Let $K$ be an infinite field of characteristic $p>0$ and let $n \ge r$ be positive integers. Let $\lambda,\mu \in \we^{+}(n,r)$ be partitions such that $\lambda=(\lambda_1,...,\lambda_m)$ and $\mu =(\mu_1, \mu_2)$, where $\lambda_m \neq 0$, $m \ge 2$ and $\mu_2 \le \la_1 \le \mu_1$. If $ p $ divides  all of the following integers \begin{align*}
&R(\lambda_1 -\mu_2+1,l), l=\min\{\la_2, \mu_1 -\la_1\} \\ &R(\lambda_i+1,\lambda_{i+1}), i=2,...,m-1.
\end{align*} Then the map \[\psi =\sum_{T \in \mathrm{ST}_{\la}(\mu)}\phi_T\] induces a nonzero homomorphism $\Delta(\la) \to \Delta(\mu).$ \end{theorem}
\begin{remark}\normalfont Consider the symmetric group $\mathfrak{S}_r$ on $r$ symbols. For a partition $\la$ of $r$, let $\Sp(\la)$ be the corresponding Specht module defined in Section 6.3 of \cite{Gr}. From Theorem 3.7 of \cite{CL}, we have  \[\dim \Hom_S(\Delta(\la),\Delta(\mu)) \le \dim \Hom_{\mathfrak{S}_r}(\Sp(\mu),\Sp(\la))\] for all partitions  $\la, \mu$ of $r$. (In fact we have equality if $p>2$ according to loc. cit.) Hence our Theorem 3.1 may be considered as a non vanishing result for homomorphisms between Specht modules.
\end{remark}

\begin{remarks}\normalfont Here we make some comments concerning the inequalities $n \ge r$, $m \ge 2$ and $\mu_2 \le \la_1 \le \mu_1$ in the statement of the above theorem.

(1) The assumption $n \ge r$ is needed so that the Weyl modules $\Delta(\la), \Delta(\mu)$ are nonzero. As is usual with such results, it turns out that this
assumption may be relaxed to $n \ge m$, since $m$ is the number of parts of the partition $\la$. This follows from the proof of the theorem to be given in Section 5.

It is well known that if $\Hom_S(\Delta(\la), \Delta(\mu)) \neq 0,$ then $\la \trianglelefteq \mu$ in the dominance ordering, meaning in particular that $\la_1\le \mu_1$.

If $m=1$, then by the previous remark,  $\Hom_S(\Delta(\la), \Delta(\mu)) = 0$, unless $\mu = \la $, in which case  $\Hom_S(\Delta(\la), \Delta(\mu)) =K$ by \cite{Jan}, the analogue for Weyl modules of II.2.8 Proposition.

(2) In the above remarks, the corresponding inequalities were needed to avoid trivial situations. The nature of the assumption $\mu_2 \le \la_1$ is different. There are cases where nonzero homomorphisms $\Delta(\la) \to \Delta(\mu)$ exist if $\mu_2>\la_1$. For example, let $n=3$, $p=2$, $\la=(2,2,2)$  and $\mu=(3,3)$. One may check that the map $\phi_T$, where $T=	\begin{matrix*}[l]
1^{(2)} 2\\
2 3^{(2)}
\end{matrix*},$ induces a nonzero map $\Delta(\la) \to \Delta(\mu)$. It would be interesting to find general results.
The main point for us of the assumption $\mu_2 \le \la_1$ is that every tableau $T$ in Lemma 2.3 is standard.

(3) If $\la_1=\mu_1$, then $R(\la_1-\mu_2+1,l)=0$ and the first divisibility condition of the theorem holds for all $p$. The remaining divisibility conditions are exactly those for which we have $\Hom_{S'}(\Delta(\la_2,...,\la_m),\Delta(\mu_2)) \neq 0,$ where $S'=S_K(n,r-\la_1)$. This follows, for example, from Theorem 3.1 of \cite{MS}. Hence in this case we have an instance of row removal which states that $\dim \Hom_S(\Delta(\la),\Delta(\mu)) = \dim \Hom_{S'}(\Delta(\la_2,...,\la_m),\Delta(\mu_2))$. See the paper by Fayers and Lyle \cite{FL}, Theorem 2.2 (stated for Specht modules), or the paper by Kulkarni \cite{Ku}, Proposition 1.2. 
\end{remarks}

For further use we note that the divisibility assumptions of Theorem 3.1 may be stated in a different way. For a positive integer $y$ let $l_{p}(y)$ be the least integer $i$ such that $p^i>y$. From James \cite{Jam}, Corollary 22.5, we have the following result. 

\begin{lemma}[\cite{Jam}]
Let $x \ge y$ be positive integers. Then $p$ divides $R(x,y)$ if and only if $p^{l_{p}(y)} $ divides $x$.
\end{lemma}

\section{Straightening}

For the proof of Theorem 3.1 we will need the following identities involving binomial coefficients. Our convention is that $\tbinom{a}{b}=0$ if $b>a$ or $b<0$. 
\begin{lemma} \;\begin{enumerate}

\item Let $a, m_1, ..., m_s$ be nonnegative integers and $m=m_1+\cdots+m_s$. 
\begin{enumerate}
\item We have \[\sum_{j_1+\cdots+j_s = a}\tbinom{m_1}{ j_1}\cdots \tbinom{m_s}{j_s}=\tbinom{m}{a},\]
where the sum ranges over all nonnegative integers $j_1,...,j_s$ such that $j_1+\cdots+j_s = a.$
\item If $m>0$, then \[\sum_{j_0+\cdots+j_s = m}(-1)^{j_0}\tbinom{m_1}{ j_1}\cdots \tbinom{m_s}{j_s}=0,\]
where the sum ranges over all nonnegative integers $j_0,...,j_s$ such that $j_0+\cdots+j_s = m.$
\end{enumerate}

\item Let $a,b,c$ be nonnegative integers such that $ b \le a$. Then
\begin{align*}&\sum_{j=0}^{c}(-1)^{c-j}\tbinom{a+j}{j}\tbinom{b}{c-j}=\tbinom{a-b+c}{c}= \sum_{j=0}^{c}(-1)^{j}\tbinom{a+c-j}{c-j}\tbinom{b}{j}.\end{align*}


\end{enumerate}
\end{lemma}
\begin{proof} (1) The identity in (a) is Vandermonde's identity. For (b) we have

\begin{align*}
\sum_{j_0+\cdots+j_s = m}(-1)^{j_0}\tbinom{m_1}{ j_1}\cdots \tbinom{m_s}{j_s}&=
\sum_{j_0=0}^{m} \; \sum_{j_1+\cdots+j_s = m-j_0}(-1)^{j_0}\tbinom{m_1}{ j_1}\cdots \tbinom{m_s}{j_s}\\&
=\sum_{j_0=0}^{m} (-1)^{j_0}\sum_{j_1+\cdots+j_s = m-j_0}\tbinom{m_1}{ j_1}\cdots \tbinom{m_s}{j_s}\\&=
\sum_{j_0=0}^{m} (-1)^{j_0}\tbinom{m}{m-j_0}\\&=
0.
\end{align*}

(2) The second identity is Lemma 2.6 of \cite{Ly1} for $q=1$. The first follows from the second with the substitution $ j \mapsto c-j$.



\end{proof}

We will also need the following explicit form of the straightening law concerning violations of standardness in the first column. 
\begin{lemma}Let $\mu=(\mu_1, \mu_2) \in \wedge^+(n,r)$, $(a_1,...,a_n) \in \wedge (n, \mu_1)$ and $(b_1,...,b_n) \in \wedge (n, \mu_2)$. \begin{enumerate} \item If $a_1+b_1 > \mu_1$, then $\begin{bmatrix*}
1^{(a_1)}   \cdots   n^{(a_n)} \\
1^{(b_1)}  \cdots  n^{(b_n)}
\end{bmatrix*}=0$. \item If $a_1+b_1 \le \mu_1$, then in $\Delta (\mu)$ we have
\begin{equation*}
\begin{bmatrix*}
1^{(a_1)}   \cdots   n^{(a_n)} \\
1^{(b_1)}  \cdots  n^{(b_n)}
\end{bmatrix*}=(-1)^{b_1}\sum_{i_2,...,i_n} \tbinom{b_2+i_2}{b_2} \cdots \tbinom{b_n+i_n}{b_n}\begin{bmatrix*}[l]
1^{(a_1+b_1)}  2^{(a_2-i_2)}  \cdots   n^{(a_n-i_n)} \\
\noindent 2^{(b_2+i_2)} \cdots n^{(b_n+i_n)}
\end{bmatrix*},\end{equation*}
where the sum ranges over all nonnegative integers $i_2, ..., i_n$ such that $i_2+\cdots+i_n=b_1$ and $i_s \le a_s$ for all $s=2,...,n$. \end{enumerate}
\end{lemma}

\begin{proof} (1) This is clear since there is no element in $\Delta(\mu)$ of weight $(\nu_1,...,\nu_n)$ satisfying $\nu_1 > \mu_1$.

(2) We proceed by induction on $b_1$, the case $b_1=0$ being clear. Suppose $b_1>0$. Consider the element $x \in D(\mu_1+b_1, \mu_2-b_1)$, where \[x=1^{(a_1+b_1)}2^{(a_2)}\cdots n^{(a_n)} \otimes 2^{(b_2)}\cdots n^{(b_n)},\] and the map
\begin{align*}
\delta : D(\mu_1+b_1, \mu_2-b_1) \xrightarrow{\Delta \otimes 1} 
D(\mu_1, b_1, \mu_2-b_1) \xrightarrow{1 \otimes \eta }
D(\mu_1, \mu_2).
\end{align*}
According to the analogue of Lemma II.2.9 of \cite{ABW} for divided powers in place of exterior powers, we have $d'_{\mu}(\delta(x))=0$ in $\Delta(\mu_1,\mu_2)$. Thus

\begin{equation*}
\begin{bmatrix*}
1^{(a_1)}   \cdots   n^{(a_n)} \\
1^{(b_1)}  \cdots  n^{(b_n)}
\end{bmatrix*}=-\sum_{j_1,...,j_n} \tbinom{b_2+j_2}{b_2} \cdots \tbinom{b_n+j_n}{b_n}\begin{bmatrix*}[l]
1^{(a_1+b_1-j_1)}  2^{(a_2-j_2)}  \cdots   n^{(a_n-j_n)} \\
\noindent 1^{(j_1)}2^{(b_2+j_2)} \cdots n^{(b_n+j_n)}
\end{bmatrix*},
\end{equation*}
where the sum ranges over all nonnegative integers $j_1, ..., j_n$ such that $j_1+\cdots+j_n=b_1$, $j_1<b_1$ and $j_s \le a_s$ for all $s=2,...,n$. Let $X$ be the right hand side of the above equality. By induction we have

\begin{align*}
X=-\sum_{j_1,...,j_n} \tbinom{b_2+j_2}{b_2} \cdots \tbinom{b_n+j_n}{b_n}(-1)^{j_1}\sum_{k_2,...,k_n} \tbinom{b_2+j_2+k_2}{b_2+j_2} \cdots \tbinom{b_n+j_n+k_n}{b_n+j_n} \\ \begin{bmatrix*}[l]
1^{(a_1+b_1)}  2^{(a_2-j_2-k_2)}  \cdots   n^{(a_n-j_n-k_n)} \\
\noindent 2^{(b_2+j_2+k_2)} \cdots n^{(b_n+j_n+k_n)}
\end{bmatrix*},
\end{align*}
where the new sum  ranges over all nonnegative integers $k_2, ..., k_n$ such that $k_2+\cdots+k_n=j_1$ and $k_s \le a_s-j_s$ for all $s=2,...,n$. Using the identities 
\[\tbinom{b_s+j_s}{b_s} \tbinom{b_s+j_s+k_s}{b_s+j_s} =\tbinom{b_s+j_s+k_s}{b_s} \tbinom{j_s+k_s}{j_s} \]
for $s=2,...,n$, we obtain 

\begin{align*}
X=-\sum_{j_1,...,j_n,k_2,...,k_n}(-1)^{j_1} \tbinom{b_2+j_2+k_2}{b_2} \cdots \tbinom{b_n+j_n+k_n}{b_n}\tbinom{j_2+k_2}{j_2} \cdots \tbinom{j_n+k_n}{j_n} \\ \begin{bmatrix*}[l]
1^{(a_1+b_1)}  2^{(a_2-j_2-k_2)}  \cdots   n^{(a_n-j_n-k_n)} \\
\noindent 2^{(b_2+j_2+k_2)} \cdots n^{(b_n+j_n+k_n)}
\end{bmatrix*}.
\end{align*}
The coefficient $c$ of \[\begin{bmatrix*}[l]
1^{(a_1+b_1)}  2^{(a_2-i_2)}  \cdots   n^{(a_n-i_n)} \\
\noindent 2^{(b_2+i_2)} \cdots n^{(b_n+i_n)}
\end{bmatrix*} \] in the right hand side of the above equation is equal to \[-\sum_{\substack{j_1,...,j_n,k_2,...,k_n \\ j_s+k_s=i_s}}(-1)^{j_1} \tbinom{b_2+j_2+k_2}{b_2} \cdots \tbinom{b_n+j_n+k_n}{b_n}\tbinom{j_2+k_2}{j_2} \cdots \tbinom{j_n+k_n}{j_n},\]
where the sum is restricted over those $j_1,...,j_n$ and $k_2,...,k_n$ that satisfy the additional conditions $j_s+k_s=i_s$ for all $s=2,...,n.$ Hence \begin{align*}c&=-\sum_{j_1,...,j_n}(-1)^{j_1} \tbinom{b_2+i_2}{b_2} \cdots \tbinom{b_n+i_n}{b_n}\tbinom{i_2}{j_2} \cdots \tbinom{i_n}{j_n} \\&
=-\tbinom{b_2+i_2}{b_2} \cdots \tbinom{b_n+i_n}{b_n}\sum_{j_1,...,j_n}(-1)^{j_1} \tbinom{i_2}{j_2} \cdots \tbinom{i_n}{j_n}.\\&
\end{align*} Remembering that in the last sum we have $j_1<b_1$,  Lemma 4.1(1)(b) yields \[\sum_{j_1,...,j_n}(-1)^{j_1} \tbinom{i_2}{j_2} \cdots \tbinom{i_n}{j_n}=0-(-1)^{b_1}.\]
Thus $c=(-1)^{b_1}\tbinom{b_2+i_2}{b_2} \cdots \tbinom{b_n+i_n}{b_n}.$
\end{proof}

\section{Proof of the main theorem}

Consider the map $\psi \in \Hom_S(D(\la), \Delta(\mu))$ given by the sum  \[\psi = \sum_{T \in \mathrm{ST}_{\la}( \mu)}\phi_T\] in the statement of Theorem 3.1 We will show, according to Theorem 2.1, that $\psi(x)=0$ for every $x \in Im(\square_\la)$. First we look at the relations corresponding to rows 1 and 2 of $\Delta(\la)$.

\textbf{Relations from rows 1 and 2}

Let $x=1^{(\la_1)}\otimes 1^{(t)}2^{(\la_2-t)}\otimes3^{(\la_3)}\cdots m^{(\la_m)} \in Im(\square_\la),$ where $t\le \la_2$, and let  $T \in \mathrm{ST}_{\la}(\mu).$ Then $T$ is of the form \[T=\begin{matrix*}[l]
1^{(\la_1)}2^{(a_2)}   \cdots   m^{(a_m)} \\
2^{(b_2)}  \cdots  m^{(b_m)} \end{matrix*} \in \mathrm{ST}_{\la}(\mu),\]
where the $a_i, b_i$ satisfy the conditions of Lemma 2.3. Using the definition of $\phi_T$ from 2.4, we have
\begin{align*}\phi_T(x)=
\sum_{i\le t}\tbinom{\la_1+i}{i}\begin{bmatrix*}[l]
1^{(\la_1+i)}2^{(a_2-i)}3^{(a_3)}  \cdots   m^{(a_m)} \\
1^{(t-i)}2^{(\la_2-t-a_2+i)}3^{(b_3)}  \cdots  m^{(b_m)}
\end{bmatrix*}.
\end{align*} If $(\la_1+i)+(t-i) \ge \mu_1$, then by the first part of Lemma 4.2 we obtain $\phi_T(x)=0$. Hence we may assume that $t \le min\{\la_2, \mu_1-\la_1\}. $ Using the second part of Lemma 4.2, we have

\begin{align*}\phi_T(x)=&
\sum_{i\le t}\tbinom{\la_1+i}{i}(-1)^{t-i} \sum_{k_2+\cdots+k_m=t-i}\tbinom{b_2-k_3-\cdots-k_m}{k_2}\tbinom{b_3+k_3}{k_3}\cdots\tbinom{b_m+k_m}{k_m}\\&
\begin{bmatrix*}[l]
1^{(\la_1+t)}2^{(a_2+k_3+\cdots+k_m)}3^{(a_3-k_3)}  \cdots   m^{(a_m-k_m)} \\
2^{(b_2-k_3-\cdots-k_m)}3^{(b_3+k_3)}  \cdots  m^{(b_m+k_m)}
\end{bmatrix*}.\end{align*}
Let $c \in K$ be the coefficient of $\begin{bmatrix*}[l]
1^{(\la_1+t)}2^{(a_2+k_3+\cdots+k_m)}3^{(a_3-k_3)}  \cdots   m^{(a_m-k_m)} \\
2^{(b_2-k_3-\cdots-k_m)}3^{(b_3+k_3)}  \cdots  m^{(b_m+k_m)}
\end{bmatrix*} $ in the right hand side of the last equation and let $k=k_3+\cdots+k_m$. Then 
\begin{align*}c&=
\left(\sum_{i=0}^{t}\tbinom{\la_1+i}{i}(-1)^{t-i}\tbinom{b_2-k}{t-k-i}\right) \tbinom{b_3+k_3}{k_3}\cdots\tbinom{b_m+k_m}{k_m}\\&=
(-1)^k\left(\sum_{i=0}^{t-k}\tbinom{\la_1+i}{i}(-1)^{t-k-i}\tbinom{b_2-k}{t-k-i}\right) \tbinom{b_3+k_3}{k_3}\cdots\tbinom{b_m+k_m}{k_m}\\&=
(-1)^{k}\tbinom{\la_1-b_2+t}{t-k}\tbinom{b_3+k_3}{k_3}\cdots\tbinom{b_m+k_m}{k_m},
\end{align*}
where in the third equality we used the first identity of Lemma 4.1 (2). Thus 

\begin{align*}\phi_T(x)=&
\sum_{k_3,...,k_m}(-1)^k\tbinom{\la_1-b_2+t}{t-k}\tbinom{b_3+k_3}{k_3}\cdots\tbinom{b_m+k_m}{k_m}
\begin{bmatrix*}[l]
1^{(\la_1+t)}2^{(a_2+k)}3^{(a_3-k_3)}  \cdots   m^{(a_m-k_m)} \\
2^{(b_2-k)}3^{(b_3+k_3)}  \cdots  m^{(b_m+k_m)}
\end{bmatrix*},\end{align*}
where $k=k_3+\cdots+k_m$ and the sum ranges over all nonnegative integers $k_3,...,k_m$
such that  $k \le b_2$ and $k_s \le a_s$ for all $s=3,...,m$.

By summing with respect to $T \in \mathrm{ST}_{\la}({\mu})$ and using Lemma 2.3 we obtain

\begin{align}\psi(x)=&\sum_{b_2,...,b_m}\sum_{k_3,...,k_m}(-1)^k\tbinom{\la_1-b_2+t}{t-k}\tbinom{b_3+k_3}{k_3}\cdots\tbinom{b_m+k_m}{k_m}\\&
\begin{bmatrix*}[l]
1^{(\la_1+t)}2^{(a_2+k)}3^{(a_3-k_3)}  \cdots   m^{(a_m-k_m)} \\
2^{(b_2-k)}3^{(b_3+k_3)}  \cdots  m^{(b_m+k_m)}
\end{bmatrix*}\nonumber,\end{align}
where the new sum is over all nonnegative integers $b_2,...,b_m$ such that $b_i \le \la_i  (i=2,...,m)$ and $b_2+\cdots+b_m=\mu_2$. 

Fix \[[S]=\begin{bmatrix*}[l]
1^{(\la_1+t)}2^{(a_2+k)}3^{(a_3-k_3)}  \cdots   m^{(a_m-k_m)} \\
2^{(b_2-k)}3^{(b_3+k_3)}  \cdots  m^{(b_m+k_m)}
\end{bmatrix*} \in \Delta(\mu)\]
in the right hand side of (5.1) and let $q=\mu_2-(b_3+k_3)-\cdots-(b_m+k_m).$ Then $q= b_2-k$. The coefficient of $[S]$ in (5.1) is equal to 

\begin{align*}
&\sum_{k}(-1)^k\tbinom{\la_1-q-k+t}{t-k} \sum_{k_3+\cdots+k_m=k}\tbinom{b_3+k_3}{k_3} \cdots \tbinom{b_m+k_m}{k_m}\\
&=\sum_{k}(-1)^k\tbinom{\la_1-q-k+t}{t-k} \tbinom{\mu_2 -q}{k}\\&
=\tbinom{\la_1-\mu_2+t}{t}\\&
=0,
\end{align*}
where in the first equality we used Lemma 4.1(1)(a) and in the second equality we used the second identity of Lemma 4.1(2).

\textbf{Relations from rows $i$ and $i+1$ ($i>1$).}

This computation is similar to the previous one but simpler as there is no straightening. Let $y=1^{(\la_1)}\otimes\cdots\otimes {i}^{(\la_i)} \otimes i^{(t)}(i+1)^{(\la_{i+1}-t)} \otimes\cdots\otimes m^{(\la_m)} \in Im(\square_\la),$ where $ i>1$ and $t \le \la_{i+1}$. As before let  \[T=\begin{matrix*}[l]
1^{(a_1)}   \cdots   m^{(a_m)} \\
2^{(b_2)}  \cdots  m^{(b_m)} \end{matrix*} \in \mathrm{ST}_{\la}(\mu).\]
The definition of $\phi_T$ yields
\begin{align*}\phi_T(y)=
\sum_{j\le t}\tbinom{a_i+j}{j}\tbinom{b_i+t-j}{t-j}\begin{bmatrix*}[l]
1^{(\la_1)}2^{(a_2)}\cdots i^{(a_i+j)}(i+1)^{(a_{i+1}-j)}  \cdots   m^{(a_m)} \\
2^{(b_2)}\cdots i^{(b_i+t-j)} (i+1)^{(b_{i+1}-t+j)}\cdots  m^{(b_m)}
\end{bmatrix*}.\end{align*}

By summing with respect to $T \in \mathrm{ST}(\la, \mu)$ and using Lemma 2.3 we have

\begin{align}\psi(y)=
&\sum_{b_2,...,b_m}\sum_{j\le t}\tbinom{\la_i - b_i+j}{j}\tbinom{b_i+t-j}{t-j}\\&\begin{bmatrix*}[l]
1^{(\la_1)}2^{(\la_2-b_2)}\cdots i^{(\la_i-b_i+j)}(i+1)^{(\la_{i+1}-b_{i+1}-j)}  \cdots   m^{(\la_m-b_m)} \\
2^{(b_2)}\cdots i^{(b_i+t-j)} (i+1)^{(b_{i+1}-t+j)}\cdots  m^{(b_m)}\nonumber
\end{bmatrix*}\end{align}
where the new sum ranges over all nonnegative integers $b_2,...,b_m$ such that $b_i \le \la_i \; (i=2,...,m)$ and $b_2+\cdots+b_m=\mu_2$.

Fix \[[S]=\begin{bmatrix*}[l]
1^{(\la_1)}2^{(\la_2-b_2)}\cdots i^{(\la_i-b_i+j)}(i+1)^{(\la_{i+1}-b_{i+1}-j)}  \cdots   m^{(\la_m-b_m)} \\
2^{(b_2)}\cdots i^{(b_i+t-j)} (i+1)^{(b_{i+1}-t+j)}\cdots  m^{(b_m)}\nonumber
\end{bmatrix*} \in \Delta(\mu)\]
in the right hand side of (5.3) and let $q=b_i-j$. The coefficient of $[S]$ in (5.3) is equal to

\begin{align*}\sum_{j\le t}\tbinom{\la_j - q}{j}\tbinom{t+q}{t-j} =\tbinom{\la_i+t}{t}=0,
\end{align*}
where in the first equality we used Lemma 4.1 (1)(a).

We have shown thus far that the map  $\psi = \sum_{T \in \mathrm{ST}(\la, \mu)}\phi_T$ induces a homomorphism of $S$-modules
$\bar{\psi}: \Delta(\la) \to \Delta(\mu)$ and it remains to be shown that $\bar{\psi} \neq 0$. Let $z=1^{(\la_1)}\otimes\cdots\otimes m^{(\la_m)} \in D(\la)$ and  $T \in \mathrm{ST}_{\la}(\mu).$ Then from the definition of $\phi_T$ we have $ \phi_T(x)=[T]$ and hence \[\psi(x)=\sum_{T \in \mathrm{ST}_{\la}(\mu)}[T].\] The right hand side is a  sum of distinct basis elements in $\Delta(\mu)$ (each with coefficient 1) according to Theorem 2.2 and hence nonzero. The proof is complete.

\begin{remark}\normalfont
Lyle has shown in \cite{Ly2}, Propositions 2.19 through 2.27 and subsection 3.3, that the homomorphism spaces between Specht modules corresponding to partitions $\la=(\la_1,...,\la_n), \mu=(\mu_1,\mu_2)$ of $r$ with $\mu_2 \le \la_1$, over the complex Hecke algebra $\mathcal{H}=\mathcal{H}_{\mathbb{C},q}(\mathfrak{S}_r)$ of the symmetric group $\mathfrak{S}_r$, where $q$ is a complex root of unity, are at most 1 dimensional. Furthermore she proves exactly when they are nonzero and provides a generator which turns out to correspond to the sum of all standard tableaux in $\mathrm{ST}_{\la}(\mu)$. (Note that our $\la, \mu$ are reversed). In the statement of Theorem 3.1 a similar map is considered and there are some technical similarities between the proof of our main result and \cite{Ly2}. However, we show in the next section, our modular homomorphism spaces may have dimension greater than 1.
\end{remark}
\section{Homomorphism spaces of dimension greater than 1}

As mentioned in the Introduction, the first examples of Weyl modules $\Delta(\la),\Delta(\mu)$ such that $\dim\Hom_S(\Delta(\la),\Delta(\mu))>1$ were obtained  by Dodge \cite{Do}. More examples were found by Lyle \cite{Ly1}, in fact in the $q$-Schur algebra setting.  The purpose of this section is to observe that the homomorphism spaces of Theorem 3.1 may have dimension $>1$, see Corollary 6.2 and Example 6.4 below.



%


We recall the following special case of the classical nonvanishing result of Carter and Payne \cite{CP}. Here boxes are raised between consecutive rows. See \cite{Ma}, 1.2 Lemma, for a proof of this particular case in our context.
\begin{proposition}[\cite{CP}]
Let $n \ge r$. Let $\la, \mu \in \we^+(n,r)$ such that for some some $d>0$ we have $\mu=(\la_1+d,\la_2-d,\la_3,...,\la_m)$, where $\la=(\la_1,...,\la_m)$. Suppose $p$ divides $R(\la_1-\la_2+d+1,d)$. Then the map \begin{align*}
\alpha : D(\la_1,\la_2,...,\la_m) &\xrightarrow{1 \otimes \Delta \otimes 1 } D(\la_1,d, \la_2-d,...,\la_m)\\& \xrightarrow{\eta \otimes 1} D(\la_1+d,\la_2-d,...,\la_m),
\end{align*}where $\Delta : D(\la_2) \to D(d,\la_2-d)$ is the indicated diagonalization and  $\eta : D(\la_1, d) \to D(\la_1+d)$  and the indicated multiplication, induces a nonzero homomorphism $\Delta(\la) \to \Delta(\mu)$.
\end{proposition}
The main result of this section is the following.
\begin{cor}
Let $n \ge r$.	Let $\lambda,\mu \in \we^{+}(n,r)$ such that $\lambda=(\lambda_1,...,\lambda_m), \lambda_m \neq 0,  m \ge 3$ and  $\mu=(\mu_1, \mu_2)$. Define $d=\mu_1-\la_1$ and assume $0<d \le \la_2 - \la_3$ and $\mu_2 \le \la_1$. If $ p $ divides  all of the following integers \begin{enumerate}
\item $R(\lambda_1 -\mu_2+1,d),$ 
\item $R(\lambda_i+1,\lambda_{i+1}), \; i=2,...,m-1,$
\item $R(\lambda_1 -\la_2+d+1,d), $
\item $R(\lambda_2-d+1,\lambda_3),$
\end{enumerate}
then the dimension of the $K$-vector space $\Hom_S(\Delta(\la), \Delta(\mu))$ is at least 2.
\end{cor}
\begin{proof}
By the first two divisibility conditions, the map \[\psi_1 =\sum_{T \in \mathrm{ST}_{\la}(\mu)}\phi_T : D(\la) \to D(\mu)\] induces a nonzero homomorphism $\bar{\psi_1}: \Delta(\la) \to \Delta(\mu)$ according to Theorem 3.1.

Next consider the following maps \begin{align*}
\alpha : D(\la_1,\la_2,...,\la_m) \to D(\la_1+d,\la_2-d,...,\la_m)
\end{align*} 
as in Proposition 6.1 and 
\begin{align*}
\beta : D(\la_1+d, \la_2-d,...,\la_m) &\xrightarrow{1 \otimes \eta' } D(\la_1+d, \la_2-d+\la_3+\cdots+\la_m)
\end{align*}
where  $\eta' : D(\la_2-d,...,\la_m) \to D(\la_2-d+\la_3+\cdots +\la_m)$ are the indicated multiplications.

Under assumption (3), we have that $\alpha$ induces a nonzero map \[\bar{\alpha}: \Delta(\la) \to D(\la_1+d,\la_2-d,...,\la_m)\]
according to Proposition 6.1

Under assumptions (2) and (4), we have that $\beta$ induces a nonzero map \[\bar{\beta}: \Delta(\la_1+d, \la_2-d,...,\la_m) \to \Delta(\la_1+d, \la_2-d+\la_3+\cdots +\la_m)\]
according to Theorem 2.1

Consider the composition $\bar{\psi_2} =\bar{\beta}\bar{\alpha} : \Delta(\la) \to \Delta(\mu)$ depicted below, where Weyl modules are indicated by the diagrams of the corresponding partitions.

\begin{center}
\begin{tikzpicture}[scale=0.6]
\draw (0,0) rectangle (3.5,0.5);
\draw (4,0.2) node {$\la_1$};
\draw (4,-0.3) node {$\la_2$};
\draw(0,-0.5) rectangle (2,0);
\draw[fill=gray!30] (1.3,-0.5) rectangle (2,0);
\draw (0,-1) rectangle (1,-0.5);
\draw (0.5,-1.5) node {$\cdots$};
\draw (0,-2) rectangle (0.7,-2.5);
\draw (5,-0.5) node {$\xrightarrow{\bar{\alpha}}$};
\draw (5.7,0) rectangle (9.5,0.5);
\draw[fill=gray!30] (8.8,0) rectangle (9.5,0.5);
\draw (5.7,-0.5) rectangle (7,0);
\draw (5.7,-1) rectangle (6.7,-0.5);
\draw (6.2,-1.5) node {$\cdots$};
\draw (5.7,-2) rectangle (6.4,-2.5);
\draw (12.2,-0.5) node {$\xrightarrow{\bar{\beta}}$};
\draw (13.2,0) rectangle (17.2,0.5);
\draw (13.2,-0.5) rectangle (16.2,0);
\draw (10.8,0.2) node {$\la_1+d$};
\draw (10.8,-0.3) node {$\la_2-d$};
\draw (18.2,0.2) node {$\la_1+d$};
\draw (17.8,-0.3) node {$\mu_2$};
\end{tikzpicture}
\end{center}

It remains to be shown that the homomorphisms $\bar\psi_1, \bar\psi_2$ are linearly independent. Let $z=d'_{\la}(1^{(\la_1)}\otimes\cdots\otimes m^{(\la_m)}) \in \Delta(\la)$. From the definitions of the maps we have 
\begin{equation*}
\bar\psi_1(z)=\sum_{T \in \mathrm{ST}_{\la}(\mu)}[T]\end{equation*}
and
\begin{equation*}\bar\psi_2(z)=\begin{bmatrix*}[l]
1^{(\la_1)}2^{(d)} \\
2^{(\la_2-d)}  \cdots  m^{(\la_m)}\end{bmatrix*}.\end{equation*} It is clear that $\begin{matrix*}[l]
1^{(\la_1)}2^{(d)} \\
2^{(\la_2-d)}  \cdots  m^{(\la_m)}\end{matrix*} \in \mathrm{ST}_{\la}(\mu).$ Since $\la_3>0$, the set $\mathrm{ST}_{\la}(\mu)$ contains at least two elements. Hence from the above equations and Theorem 2.2 it follows that the maps $\bar\psi_1, \bar\psi_2$ are linearly independent.
\end{proof}

\begin{remark}\normalfont The assumptions of Corollary 6.2 imply that for the corresponding Specht modules we have
$ \dim \Hom_{\mathfrak{S}_r}(\Sp (\mu),\Sp (\la)) \ge 2.$
See Remark 3.2.
\end{remark}
\begin{example}\normalfont Let $p$ be a prime and $a$ an integer such that $a \ge (p^2+1)(p-1)$ and 
\[a \equiv p-2 \mod p^2.\]
Consider the following partitions \begin{align*}&\la=(a, 2p-1, (p-1)^{p^2}), \\& \mu=(a+p, (p^2+1)(p-1)),\end{align*}
where $p-1$ appears $p^2$ times as a row in $\la$. Using Lemma 3.3 it easily follows that the assumptions (1) - (4) of  Corollary 6.2 are satisfied. For example, we have \[\la_1-\mu_{2}+1 \equiv p-2-(p^2+1)(p-1)+1 \equiv 0 \mod p^2\] and hence by Lemma 3.3, $d=p$ divides  $R(\la_1 - \mu_{2}+1,d)$ which is assumption (1). Thus $\dim\Hom_S(\Delta(\la), \Delta(\mu)) \ge 2.$ \footnote{We note that for fixed $p$, it follows from the main result of \cite{MS2} that the dimension of $\Hom_S(\Delta(\la), \Delta(\mu))$ does not depend on $a$. For $p=3$, this means that $\dim\Hom_S(\Delta(\la), \Delta(\mu))=2$ for all $a$, see \cite{MS2}, Example 2.4.}

For $p=2$ the least $a$ that satisfies the above requirements is $a=8$ and thus we have the partitions $\la=(8,3,1,1,1,1), \mu=(10,5)$. This pair appears in Example 4, subsection 2.3, of Lyle's paper \cite{Ly2} which prompted us to consider Corollary 6.2 and in particular the composition $\bar{\psi_2} =\bar{\beta}\bar{\alpha} : \Delta(\la) \to \Delta(\mu)$.

\end{example}


%

\end{document}